\documentclass[12pt]{article}
\usepackage[margin = 20mm]{geometry}
\usepackage{amsmath}
\usepackage{amsfonts}
\usepackage{graphicx}
\usepackage{mathrsfs}
\usepackage{amssymb}
\usepackage{amsthm}
\usepackage{cases}
\usepackage{enumerate}
\usepackage{tocloft}
\usepackage{setspace}
\usepackage{abstract}
\usepackage{nicematrix}
\usepackage{graphicx}
\usepackage{xcolor, colortbl}
\usepackage[ruled]{algorithm2e}
\usepackage{tikz}
\usetikzlibrary{shapes.geometric}
\usetikzlibrary{graphs}
\usetikzlibrary{graphs.standard}
\usepackage[backend = biber, doi = false, url = false, isbn = false, maxbibnames=222]{biblatex}
\renewbibmacro{in:}{}
\newtheorem{thm}{Theorem}[section]
\newtheorem{cor}[thm]{Corollary}
\newtheorem{lem}[thm]{Lemma}

\theoremstyle{definition}
\newtheorem{defin}[thm]{Definition}

\def\eref#1{$(\ref{#1})$}
\def\sref#1{\S$\ref{#1}$}
\def\lref#1{Lemma~$\ref{#1}$}
\def\tref#1{Theorem~$\ref{#1}$}
\def\cyref#1{Corollary~$\ref{#1}$}

\def\dref#1{Definition~$\ref{#1}$}

\def\Id{\epsilon}
\def\F{\mathbb{F}}
\def\K{\mathbb{K}}
\def\R{\mathcal{R}}
\def\N{\mathcal{N}}
\def\L{\mathcal{L}}

\def\T{\mathcal{T}}

\def\Q{\mathcal{Q}}
\def\Atp{\textnormal{Atp}}
\def\Aut{\textnormal{Aut}}
\def\Sym{\textnormal{Sym}}
\def\Alt{\textnormal{Alt}}
\def\GL{\textnormal{GL}}
\def\GammaL{\Gamma\textnormal{L}}
\def\AGL{\textnormal{AGL}}
\def\PGL{\textnormal{PGL}}
\def\PSL{\textnormal{PSL}}
\def\Sym{\textnormal{Sym}}
\def\Gal{\textnormal{Gal}}
\def\PGammaL{\textnormal{P}\Gamma\textnormal{L}}
\def\AGammaL{\textnormal{A}\Gamma\textnormal{L}}
\def\AGamma2L{\textnormal{A}\Gamma^2\textnormal{L}}

\renewcommand{\geq}{\geqslant}
\renewcommand{\leq}{\leqslant}
\renewcommand{\emptyset}{\varnothing}
\title{Isotopisms of quadratic quasigroups}
\author{Jack Allsop\\
	\small School of Mathematics\\[-0.5ex]
	\small Monash University\\[-0.5ex]
	\small Vic 3800, Australia\\
	\small\tt jack.allsop@monash.edu}
\date{}

\begin{document}
	\maketitle
	
	\begin{abstract}
		A quasigroup is a pair $(Q, \cdot)$ where $Q$ is a non-empty set and $\cdot$ is a binary operation on $Q$ such that for every $(u, v) \in Q^2$ there exists a unique $(x, y) \in Q^2$ such that $u \cdot x = v = y \cdot u$. Let $q$ be an odd prime power, let $\F_q$ denote the finite field of order $q$, and let $\R_q$ denote the set of non-zero squares in $\F_q$. Let $\{a, b\} \subseteq \F_q$ be such that $\{ab, (a-1)(b-1)\} \subseteq \R_q$. Let $\Q_{a, b}$ denote the quadratic quasigroup $(\F_q, *_{a, b})$ where $*_{a, b}$ is defined by
		\[
		x*_{a, b}y = \begin{cases}
			x+a(y-x) & \text{if } y-x \in \R_q, \\
			x+b(y-x) & \text{otherwise}.
		\end{cases}
		\]
		The operation table of a quadratic quasigroup is a quadratic Latin square. Recently, it has been determined exactly when two quadratic quasigroups are isomorphic and the automorphism group of any quadratic quasigroup has been determined. In this paper, we extend these results. We determine exactly when two quadratic quasigroups are isotopic and we determine the autotopism group of any quadratic quasigroup.
		In the process, we count the number of $2 \times 2$ subsquares in quadratic Latin squares.
	\end{abstract}
	
	\section{Introduction}\label{s:intro}
	
	Throughout this paper, $p$ will denote an odd prime, $d$ will denote a positive integer, and $q = p^d$ will be an odd prime power.
	
	A \emph{quasigroup} is a pair $(Q, \cdot)$ where $Q$ is a non-empty set and $\cdot$ is a binary operation on $Q$ such that for every $(u, v) \in Q^2$ there exists a unique $(x, y) \in Q^2$ such that $u \cdot x = v = y \cdot u$. We will often refer to the quasigroup $(Q, \cdot)$ simply by $Q$. If $Q$ is finite, then $|Q|$ is the \emph{order} of $Q$. We are only concerned with finite quasigroups in this paper.
	
	Let $n$ be a positive integer. A \emph{Latin square} of order $n$ is an $n \times n$ matrix of $n$ symbols, each of which occurs exactly once in each row and column. In this paper, we will assume that the rows and columns of any Latin square are indexed by its symbol set. With this convention, Latin squares are combinatorial equivalents of finite quasigroups, in the sense that the operation table of any quasigroup of order $n$ is a Latin square of order $n$ and any Latin square of order $n$ can be interpreted as the operation table of a quasigroup of order $n$.
	
	Let $\F_q$ denote the finite field of order $q$, let $\R_q$ denote the set of quadratic residues in the multiplicative group $\F_q^*$, and let $\N_q$ denote the set of quadratic non-residues in $\F_q^*$. 
	Let $\{a, b\} \subseteq \F_q$ be such that $\{ab, (a-1)(b-1)\} \subseteq \R_q$. Define a binary operation $* = *_{a, b}$ on $\F_q$ by 
	\[
	x*y = \begin{cases}
		x & \text{if } x=y,\\
		x+a(y-x) & \text{if } y-x \in \R_q,\\
		x+b(y-x) & \text{if } y-x \in \N_q.
	\end{cases}
	\]
The pair $(\F_q, *)$ defines a quasigroup, called a \emph{quadratic quasigroup}, which is denoted by $\Q_{a, b}$. The operation table of a quadratic quasigroup is a \emph{quadratic Latin square}. The operation table of $\Q_{a, b}$ is denoted by $\L[a, b]$. Note that the operation $*_{a, b}$ can be defined for any $\{a, b\} \subseteq \F_q$, but the condition $\{ab, (a-1)(b-1)\} \subseteq \R_q$ is necessary and sufficient for $(\F_q, *_{a, b})$ to be a quasigroup~\cite{orthgraph}.
	
	Quadratic quasigroups have previously been used to construct perfect $1$-factorisations~\cite{nu4, p1f16, cycatom}, anti-perfect $1$-factorisations~\cite{quadcycs}, mutually orthogonal Latin squares~\cite{orthls, orthgraph}, atomic Latin squares~\cite{cycatom}, anti-atomic Latin squares~\cite{quadcycs}, $N_2$ Latin squares~\cite{quadcycs}, Latin trades~\cite{bitrades}, Falconer varieties~\cite{nu4}, and maximally non-associative quasigroups~\cite{mnaq, numquad, maxnonass}.
	
	Let $X$ be a set. Let $\Sym(X)$ denote the group of permutations of $X$ and let $\Alt(X)$ denote the group of even permutations of $X$. Let $\Id \in \Sym(X)$ denote the identity permutation of $X$. Whenever we use the symbol $\Id$, the set that it is acting on will be clear from context. 
	In this paper, $f \circ g$ denotes the function $x \mapsto f(g(x))$. 
	
	Let $(Q, *)$ and $(Q', \cdot)$ be quasigroups. If there exist bijective maps $\alpha, \beta, \gamma : Q \to Q'$ such that $\alpha(x) \cdot \beta(y) = \gamma(x*y)$ for all $\{x, y\} \subseteq Q$, then $Q$ and $Q'$ are \emph{isotopic} and $(\alpha, \beta, \gamma)$ is an \emph{isotopism} from $Q$ to $Q'$.	
	If there exists a bijective map $\theta : Q \to Q'$ such that $\theta(x) \cdot \theta(y) = \theta(x*y)$ for all $\{x, y\} \subseteq Q$, then $Q$ and $Q'$ are \emph{isomorphic} and $\theta$ is an \emph{isomorphism} from $Q$ to $Q'$. An \emph{autotopism} of $Q$ is an isotopism from $Q$ to itself and we denote the group of autotopisms of $Q$ by $\Atp(Q)$. An \emph{automorphism} of $Q$ is an isomorphism from $Q$ to itself and we denote the group of automorphisms of $Q$ by $\Aut(Q)$. We will also refer to a triple $(\theta, \theta, \theta) \in \Atp(Q)$ as an automorphism of $Q$. 
	
Let $n$ be a positive integer and let $(Q, *)$ be a quasigroup of order $n$. We can identify $Q$ with the set of $n^2$ triples $\{(x, y, x*y) : \{x, y\} \subseteq Q\}$.
A \emph{parastrophe} of $Q$ is a quasigroup obtained from $Q$ by uniformly permuting the elements of each of its triples. Every quasigroup has six, not necessarily distinct, parastrophes.
The notions of isotopisms, isomorphisms, and parastrophes extend naturally to Latin squares.
	
	Dr\'{a}pal and Wanless~\cite{quadiso} determined the automorphism groups of quadratic quasigroups and also determined exactly when two quadratic quasigroups are isomorphic. Before stating their results, we need to give some definitions. 
	
	Let $\mathbb{K} \leq \mathbb{F}_q$. A permutation $\psi \in \Sym(\F_q)$ is \emph{additive} if $\psi(x+y) = \psi(x)+\psi(y)$ for all $\{x, y\} \subseteq \F_q$ and $\psi$ is \emph{$\K$-linear} if $\psi$ is additive and $\psi(kx) = k\psi(x)$ for all $k \in \K$ and $x \in \F_q$. For $c \in \F_q$, let $\tau_c : \F_q \to \F_q$ be defined by $x \mapsto x+c$ and let $\lambda_c : \F_q \to \F_q$ be defined by $x \mapsto cx$. Define $\T_q = \{\tau_c : c \in \F_q\}$. We can identify the general linear group $\GL_d(p)$ with the set of additive permutations in $\Sym(\F_q)$ and we can identify the affine general linear group $\AGL_d(p)$ with the set of permutations in $\Sym(\F_q)$ of the form $\tau_c \circ \sigma$ where $\sigma \in \GL_d(p)$ and $c \in \F_q$. 
	
	Denote the automorphism group of $\F_q$ by $\Aut(\F_q)$ and denote the group of automorphisms of $\F_q$ which fix a subfield $\K$ of $\F_q$ pointwise by $\Gal(\F_q | \K)$. Define $\AGamma2L_1(\F_q|\K)$ to be the group of all mappings of the form $x \mapsto \nu\theta(x)+\mu$ for some $\nu \in \R_q$, $\mu \in \F_q$, and $\theta \in \Gal(\F_q|\K)$. Suppose that $|\K| = \zeta^2$ for some integer $\zeta$. Then we can define $\AGammaL_1^{\textnormal{tw}}(\F_q|\K)$ to be the group consisting of all mappings in $\AGamma2L_1(\F_q | \K)$ along with all mappings of the form $x \mapsto \nu\theta(x^\zeta)+\mu$ for some $\nu \in \N_q$, $\mu \in \F_q$, and $\theta \in \Gal(\F_q|\K)$.
	
	The following two theorems were shown in~\cite{quadiso}.
	
	\begin{thm}\label{t:quadiso}
		Quadratic quasigroups $\Q_{a, b}$ and $\Q_{a', b'}$ of order $q$ are isomorphic if and only if there exists some $\theta \in \Aut(\F_q)$ such that $\{a, b\} = \{\theta(a'), \theta(b')\}$.
	\end{thm}
	
	\begin{thm}\label{t:quadaut}
		Let $\Q = \Q_{a, b}$ be a quadratic quasigroup of order $q$ and let $\K$ be the subfield of $\F_q$ generated by $a$ and $b$. Then $\Aut(\Q) = \AGamma2L_1(\F_q | \K)$ up to these exceptions:
		\begin{itemize}
			\item If $a=b$, then $\Aut(\Q) \cong \AGL_k(\K)$ where $k = [\F_q : \K]$. Every automorphism of $\Q$ is of the form $\tau_\mu \circ \sigma$ for some $\mu \in \F_q$ and some $\K$-linear map $\sigma \in \Sym(\F_q)$.
			\item If $|\mathbb{K}| = \zeta^2$ for some integer $\zeta$ and $b=a^\zeta$, then $\Aut(\Q) = \AGammaL_1^{\textnormal{tw}}(\F_q | \K)$.
			\item If $q=7$ and $\{a, b\} = \{3, 5\}$, then $\Aut(\Q) \cong \PSL_2(7)$.
		\end{itemize}
	\end{thm}
	
	We prove the analogous result to \tref{t:quadiso} for isotopisms and to \tref{t:quadaut} for autotopisms.
	
	\begin{thm}\label{t:quadisotop}
		Quadratic quasigroups $\Q_{a, b}$ and $\Q_{a', b'}$ of order $q$ are isotopic if and only if one of the following holds:
		\begin{itemize}
			\item There exists some $\theta \in \Aut(\F_q)$ such that $\{a, b\} = \{\theta(a'), \theta(b')\}$,
			\item $a=b$ and $a'=b'$.
		\end{itemize}	
	\end{thm}
	
	\begin{thm}\label{t:quadatp}
		Let $\Q = \Q_{a, b}$ be a quadratic quasigroup of order $q$. 
		\begin{itemize}
			\item If $a=b$, then $\Atp(\Q)$ is isomorphic to the semi-direct product $\T_q^2 \rtimes \GL_d(p)$ with multiplication defined by $((\tau_u, \tau_v), \theta) \cdot ((\tau_{w}, \tau_{z}), \phi) = ((\tau_{\phi^{-1}(u) + w}, \tau_{\phi^{-1}(v)+z}), \theta \circ \phi)$. Every autotopism of $\Q$ is of the form $(\lambda_{1-a}^{-1} \circ \theta \circ \tau_u \circ \lambda_{1-a}, \lambda_{a}^{-1} \circ \theta \circ \tau_v \circ \lambda_a, \theta \circ \tau_{u+v})$ for some $\theta \in \GL_d(p)$ and $\{u, v\} \subseteq \F_q$.
			\item If $a \neq b$, then every autotopism of $\Q$ is an automorphism and so $\Atp(\Q) \cong \Aut(\Q)$.
		\end{itemize}
	\end{thm}
	
	Let $\Q = \Q_{a, b}$ and $\Q' = \Q_{a', b'}$ be quadratic quasigroups with $a \neq b$ and $a' \neq b'$. \tref{t:quadisotop} tells us that $\Q$ and $\Q'$ are isotopic if and only if they are isomorphic. This statement was predicted in~\cite{nu4} in the case where $\Q$ and $\Q'$ are of prime order. 
	
	\tref{t:quadisotop} and \tref{t:quadatp} are easy to verify in the case where $a=b$. Let $\Q = \Q_{a, a}$ and $\Q' = \Q_{a', b'}$ be quadratic quasigroups of order $q$. From~\cite[Theorem $1.3$]{quadiso}, we know that $\Q$ is isotopic to the group $(\F_q, +)$ and if $a' \neq b'$, then $\Q'$ is not isotopic to any group. Therefore, $\Q$ and $\Q'$ are isotopic if and only if $a'=b'$. This proves \tref{t:quadisotop} in the case where $a=b$. The autotopism group of an abelian group $(G, +)$ is isomorphic to the semidirect product $G^2 \rtimes \Aut(G)$ with multiplication defined by $((g_1, g_2), \theta) \cdot ((h_1, h_2), \phi) = ((\phi^{-1}(g_1)+h_1, \phi^{-1}(g_2)+h_2), \theta \circ \phi)$. Furthermore, every autotopism of $G$ is of the form $(\alpha, \beta, \gamma)$ where $\alpha(x) = \theta(x+u)$, $\beta(x) = \theta(x+v)$, and $\gamma(x) = \theta(x+u+v)$ for some $\{u, v\} \subseteq G$ and $\theta \in \Aut(G)$~\cite{grpatp}. The claim of \tref{t:quadatp} in the case where $a=b$ now follows, since $(\lambda_{1-a}, \lambda_a, \Id)$ is an isotopism from $\Q$ to $(\F_q, +)$ and the automorphism group of $(\F_q, +)$ is $\GL_d(p)$. It remains to prove \tref{t:quadisotop} and \tref{t:quadatp} in the case where $a \neq b$.
	
	\medskip
	
	In order to prove \tref{t:quadisotop} and \tref{t:quadatp}, we will need to use some results about the structure of quadratic Latin squares. Let $L$ be a Latin square with symbol set $S$. An \emph{intercalate} of $L$ is a subset $\{r_1, r_2\} \subseteq S$ consisting of distinct rows and a subset $\{c_1, c_2\} \subseteq S$ consisting of distinct columns such that $L_{r_1, c_1} = L_{r_2, c_2}$ and $L_{r_1, c_2} = L_{r_2, c_1}$. 
	There has been lots of work done counting intercalates in Latin squares~\cite{cycstrucrandom, interceverywhere, maxinterc, N2rare, substructureLS, KS, manysubsq, uniqueinterc}. Particular focus has been given to Latin squares that have no intercalates. Such squares are called $N_2$. Asymptotically almost all Latin squares of order $n$ contain $\Theta(n^2)$ intercalates and so $N_2$ Latin square are rare~\cite{N2rare, substructureLS}. In~\cite{quadcycs}, it was determined exactly when a quadratic Latin square is $N_2$.
	
	\begin{thm}\label{t:n2}
		Let $\mathcal{L} = \mathcal{L}[a, b]$ be a quadratic Latin square of order $q$. If $b \in \{-a, 2-a, a/(2a-1)\}$, then $q \equiv 1 \bmod 4$ and $\L$ contains an intercalate. If $b \notin \{-a, 2-a, a/(2a-1)\}$, then $\L$ contains an intercalate if and only if
		\begin{equation}\label{e:n2cond}
				(2ab-a-b)(a+b)(a-1) \in \R_q \text{ and } \{2(a+b-2)(a-1), 2a(a+b)\} \subseteq \N_q.
		\end{equation}
	\end{thm}
	
	We generalise \tref{t:n2} by counting the number of intercalates in quadratic Latin squares.
	
	\begin{thm}\label{t:quadinterc}
		Let $\mathcal{L} = \mathcal{L}[a, b]$ be a quadratic Latin square of order $q$ and let $N$ be the number of intercalates that $\L$ contains. 
		\begin{itemize}
			\item If $b \not\in \{-a, 2-a, a/(2a-1)\}$ and \eref{e:n2cond} is not satisfied, then $N=0$.
			\item If $b \not\in \{-a, 2-a, a/(2a-1)\}$ and \eref{e:n2cond} is satisfied, then $N = q(q-1)$.
			\item If $b \in \{-a, 2-a, a/(2a-1)\}$, then
			\[
			\frac{q(q-1)(q-11q^{1/2}-38)}{32} \leq N \leq \frac{q(q-1)(q+11q^{1/2}+70)}{32}.
			\]
		\end{itemize}
	\end{thm}
	
	Let $\L = \L[a, a]$ be a quadratic Latin square of order $q$ and note that $a \not\in \{-a, 2-a, a/(2a-1)\}$ since $a \not\in \{0, 1\}$. Since $\L$ is isotopic to the Cayley table of $(\F_q, +)$, it follows that $\L$ is $N_2$. Furthermore, $2(a+a-2)(a-1) = 4(a-1)^2 \in \R_q$, thus \eref{e:n2cond} is not satisfied. Therefore, \tref{t:quadinterc} is true when $a=b$. It remains to verify \tref{t:quadinterc} in the case where $a \neq b$.
	
	Let $n$ be a positive integer. The maximum number of intercalates in a Latin square of order $n$ is $n^2(n-1)/4$ and there is a Latin square of order $n$ that achieves this bound if and only if $n$ is a power of $2$ (see, e.g.,~\cite{maxinterc}). On the other hand, there is a Latin square of order $n$ with at least $(n-1)(n-3)(n-15)/8$ intercalates~\cite{smallsubsq}. By \tref{t:quadinterc}, a Latin square $\L[a, b]$ of order $q$ with $b \in \{-a, 2-a, a/(2a-1)\}$ (which exists whenever $9 \leq q \equiv 1 \bmod 4$) has at least $q(q-1)(q-11q^{1/2}-38)/32 = \Theta(q^3)$ intercalates, which is within a constant factor of the maximum number of intercalates that a Latin square of order $q$ can contain.
	
	A \emph{species} of Latin squares is a maximal set of Latin squares that is closed under applying isotopisms and taking parastrophes. It is known~\cite{quadcycs} that there are $7q^2/32 + O(q^{3/2})$ quadratic Latin squares of order $q$ that are $N_2$. Since the $N_2$ property of a Latin square is a species invariant, it would be useful to know the number of species that contain an $N_2$ quadratic Latin square of order $q$. \tref{t:quadisotop} allows us to deduce that there are at least $\Omega(q^2/\log(q))$ species that contain an $N_2$ quadratic Latin square of order $q$.	
	Counting problems such as this (see also~\cite{nu4}) partly motivated the study of when quadratic quasigroups are isotopic. 
	
	We now outline the structure of the remainder of the paper. In~\sref{s:interc}, we prove \tref{t:quadinterc}. After that, our next goal is to prove \tref{t:quadatp} for quadratic quasigroups $\Q_{a, b}$ with $a \neq b$. For a quasigroup $Q$ and $i \in \{1, 2, 3\}$, denote by $\Atp_i(Q)$ the projection of $\Atp(Q)$ onto the $i$-th coordinate. Let $\Q = \Q_{a, b}$ be a quadratic quasigroup of order $q$ with $a \neq b$ and let $i \in \{1, 2, 3\}$. In~\sref{s:prelim}, we prove that any autotopism of $\Q$ whose components all lie in $\AGL_d(p)$ is an automorphism of $\Q$. In \sref{s:not2trans}, we prove that $\Atp_i(\Q) \subseteq \AGL_d(p)$ if $\Atp_i(\Q)$ is not $2$-transitive and in \sref{s:2trans}, we prove that $\Atp_i(\Q) \subseteq \AGL_d(p)$ if $\Atp_i(\Q)$ is $2$-transitive. In~\sref{s:quadisotop}, we combine the results proved in \sref{s:prelim}, \sref{s:not2trans}, and \sref{s:2trans} to prove \tref{t:quadatp} and then we prove \tref{t:quadisotop}.
	Finally, we give some concluding remarks in~\sref{s:conc}.
	
	\section{Intercalates in quadratic Latin squares}\label{s:interc}
	
	In this section, we prove \tref{t:quadinterc}. We start with some definitions and simple facts.
	
	Let $L$ be a Latin square with symbol set $S$ and let $i$ and $j$ be distinct elements of $S$. The permutation mapping row $i$ of $L$ to row $j$, denoted by $r_{i, j}$, is defined by $r_{i, j}(L_{i, k}) = L_{j, k}$ for each $k \in S$. Such a permutation is called a \emph{row permutation} of $L$. Every row permutation of $L$ is a derangement. When writing a row permutation in disjoint cycle notation, every cycle is called a \emph{row cycle}. A row cycle of length two is a \emph{transposition}. A transposition in $r_{i, j}$ is equivalent to an intercalate in $L$ involving rows $i$ and $j$. 
	
	Let $\L = \L[a, b]$ be a quadratic Latin square of order $q$. Let $\chi : \F_q \to \{-1, 0, 1\}$ denote the extended quadratic character on $\F_q$. Define a map $\varphi = \varphi[a, b] : \F_q \to \F_q$ by
	\[
	\varphi(x) = \begin{cases}
		ax & \text{if } \chi(x)=1,\\
		bx & \text{otherwise}.
	\end{cases}
	\]
	Then $x*_{a, b}y = x+\varphi(y-x)=\L_{x, y}$ for all $\{x, y\} \subseteq \F_q$. Let $\{i, j\} \subseteq \F_q$ with $i \neq j$. The row permutation $r_{i, j}$ of $\L$ satisfies
	\begin{equation}\label{e:rp}
		r_{i, j} = \tau_j \circ \varphi \circ \tau_{i-j} \circ \varphi^{-1} \circ \tau_{-i}.
	\end{equation}
	
	The \emph{cycle structure} of a permutation is a sorted list of the lengths of its cycles. The following lemma regarding row permutations of quadratic Latin squares is a consequence of the large automorphism groups of quadratic quasigroups.
	
	\begin{lem}\label{l:quadrowperms}
		Let $\L = \L[a, b]$ be a quadratic Latin square of order $q$ and let $i$ and $j$ be distinct elements of $\F_q$.
		\begin{enumerate}[(i)]
			\item If $q \equiv 3 \bmod 4$, then the row permutation $r_{i, j}$ of $\L$ has the same cycle structure as the row permutation $r_{0,1}$ of $\L$.
			\item If $q \equiv 1 \bmod 4$, then the row permutation $r_{i, j}$ of $\L$ has the same cycle structure as the row permutation $r_{0,1}$ of $\L$ if $\chi(j-i)=1$ and it has the same cycle structure as the row permutation $r_{0,1}$ of $\L[b, a]$ if $\chi(j-i)=-1$.
		\end{enumerate}
	\end{lem}
	
	\lref{l:quadrowperms} can be proved by combining the arguments used in the proof of~\cite[Lemma $10$]{cycatom} with the fact that $\lambda_c$ is an isomorphism from $\L[a, b]$ to $\L[b, a]$ for any $c \in \N_q$~\cite[Lemma $7$]{cycatom}.
	
	Fix any two distinct rows $i$ and $j$ of a Latin square $L$. The number of intercalates of $L$ involving rows $i$ and $j$ is exactly the number of transpositions in the row permutation $r_{i, j}$ (or $r_{j, i}$) of $L$. Therefore, in order to count the number of intercalates in a Latin square, it suffices to count the number of transpositions in each of its row permutations. Now let $\L = \L[a, b]$ be a quadratic Latin square of order $q$. By \lref{l:quadrowperms}, to count the number of transpositions in any row permutation of $\L$, it suffices to count the number of transpositions in the row permutation $r_{0, 1}$ of $\L$ and if $q \equiv 1 \bmod 4$, also the number of transpositions in the row permutation $r_{0, 1}$ of $\L[b, a]$. So, our goal in this section is to count the number of transpositions in the row permutations $r_{0, 1}$ of quadratic Latin squares. To do this, we employ the method used in~\cite{quadcycs}. We now give some definitions that come from~\cite{quadcycs}, with occasional changes in notation.
	
	Call a pair $(a, b) \in \F_q^2$ \emph{valid} if $\{ab, (a-1)(b-1)\} \subseteq \R_q$. 
	For a valid pair $(a, b) \in \F_q^2$, let $\omega[a, b]$ denote the row permutation $r_{0, 1}$ of the quadratic Latin square $\L[a, b]$ of order $q$. Since we have already verified \tref{t:quadinterc} in the case where $a=b$, we are only interested in counting the number of transpositions in permutations $\omega[a, b]$ where $(a, b) \in \F_q^2$ is valid and $a \neq b$. 
	Define
	\[
	\Gamma = \{\omega[a, b] : (a, b) \in \F_q^2 \text{ is valid and } a \neq b\}.
	\]
	
	\begin{defin}\label{d:satisfy}
	Let $\omega = \omega[a, b] \in \Gamma$, let $\varphi = \varphi[a, b]$, and let $z \in \{-1, 0, 1\}^{4}$. We say that $\omega$ \emph{satisfies} $z$ with $j \in \F_q$ if: $\omega^2(j)=j, \text{ } z_0 = \chi(j), \text{ } z_1 = \chi(\varphi^{-1}(j)-1), \text{ } z_2 = \chi(\omega(j)), \text{ and } z_3 = \chi(\varphi^{-1}(\omega(j))-1)$.
	\end{defin}
	
	Let $\omega = \omega[a, b] \in \Gamma$. If $\omega$ satisfies a sequence $z \in \{-1, 0, 1\}^4$ with some $j \in \F_q$, then we will sometimes just say that $\omega$ satisfies $z \in \{-1, 0, 1\}^4$. We denote the cycle of $\omega$ containing $0$ by $\omega_0$ and we denote the cycle of $\omega$ containing $a$ by $\omega_a$. It is a simple fact that if $\omega$ satisfies $z \in \{-1, 0, 1\}^4$ with $j \in \F_q$ and $\{0, a\} \cap \{j, \omega(j)\} = \emptyset$, then $z \in \{-1, 1\}^4$. We can determine information about the transpositions in $\omega$ by studying sequences in $\{-1, 1\}^4$ that it satisfies. 
	
	Define
	\begin{align*}
	&T_1 = \{(-1, -1, -1, 1), (-1, -1, 1, -1), (-1, 1, 1, 1), (1, -1, 1, 1)\},\\
	&T_{2, 1} = \{(-1, -1, -1, -1), (-1, -1, 1, 1), (-1, 1, 1, -1), (1, 1, 1, 1)\},\\
	&T_{2, 2} = \{(-1, -1, 1, 1), (-1, 1, -1, 1), (-1, 1, 1, -1), (1, -1, 1, -1)\}, \text{ and} \\
	&T_{2, 3} = \{(-1, -1, -1, -1), (-1, -1, 1, 1), (-1, 1, -1, 1), (-1, 1, 1, -1), (1, -1, 1, -1), (1, 1, 1, 1)\}.
	\end{align*}
	The following lemma can be inferred from~\cite{quadcycs}.
	
	\begin{lem}\label{l:sat}
		Let $\omega = \omega[a, b] \in \Gamma$ and let $(j, \omega(j))$ be a transposition in $\omega$ with $\{0, a\} \cap \{j, \omega(j)\}=\emptyset$.
		\begin{itemize}
			\item If $-a \neq b \in \R_q$, then there is a unique $z \in T_1 \cup T_{2, 1}$ such that $\omega$ satisfies $z$ with $j$ or $\omega(j)$.
			\item If $-a \neq b \in \N_q$, then there is a unique $z \in T_1 \cup T_{2, 2}$ such that $\omega$ satisfies $z$ with $j$ or $\omega(j)$.
			\item If $b=-a$, then there is a unique $z \in T_1 \cup T_{2, 3}$ such that $\omega$ satisfies $z$ with $j$ or $\omega(j)$.
		\end{itemize}
	\end{lem}
	
	Let $\omega = \omega[a, b] \in \Gamma$. If $-a \neq b \in \R_q$, then set $i=1$. If $-a \neq b \in \N_q$, then set $i=2$. Otherwise, set $i=3$. Suppose that $(j, \omega(j))$ is a transposition in $\omega$ with $\{0, a\} \cap \{j, \omega(j)\} = \emptyset$.  If $\omega$ satisfies a sequence in $T_1$ with $j$ or $\omega(j)$, then $(j, \omega(j))$ is a \emph{Type One} transposition. Otherwise, $(j, \omega(j))$ is a \emph{Type Two} transposition and satisfies a sequence in $T_{2, i}$ with $j$ or $\omega(j)$.
	The following lemma gives us a method of counting the number of Type One transpositions in a permutation in $\Gamma$.
	
	\begin{lem}\label{l:t1count}
		The number of Type One transpositions in a permutation $\omega \in \Gamma$ is equal to the number of sequences in $T_1$ that $\omega$ satisfies.
	\end{lem}
	\begin{proof}
		Suppose that $(j, \omega(j))$ is a Type One transposition in $\omega$. Then there is some $z \in T_1$ such that $\omega$ satisfies $z$ with $j$ or $\omega(j)$. \lref{l:sat} says that $\omega$ does not satisfy any $z' \in T_1 \setminus \{z\}$ with $j$ or $\omega(j)$. Therefore, we can naturally define a map $\rho$ from the set of Type One transpositions in $\omega$ to the sequences in $T_1$ that $\omega$ satisfies. It is immediate from \dref{d:satisfy} that $\rho$ is surjective. Also, the combination of~\cite[Lemma $3.5$]{quadcycs}, \cite[Lemma $3.6$]{quadcycs}, and \cite[Definition $3.7$]{quadcycs} implies that if $z' \in T_1$, then there is at most one $j' \in \F_q$ such that $\omega$ satisfies $z'$ with $j'$. Hence, $\rho$ is a bijection and the lemma follows.	
	\end{proof}
	
	We next want to prove a result analogous to \lref{l:t1count} for Type Two transpositions. Let $\omega \in \Gamma$ and let $z \in T_{2, i}$ for some $i \in \{1, 2, 3\}$. Unlike sequences in $T_1$, there may be many $j \in \F_q$ such that $\omega$ satisfies $z$ with $j$. Denote the number of such $j$ by $e(\omega, z)$. Call a sequence $z \in \{-1, 1\}^4$ \emph{even periodic} if $z_0=z_2$ and $z_1=z_3$. Let $P_{2, i} \subseteq T_{2, i}$ be the subset consisting of even periodic sequences.
	
	\begin{lem}\label{l:t2count}
		Let $\omega = \omega[a, b] \in \Gamma$. If $-a \neq b \in \R_q$, then set $i=1$. If $-a \neq b \in \N_q$, then set $i=2$. Otherwise, set $i=3$. The number of Type Two transpositions in $\omega$ is
		\[
		\sum_{z \in T_{2, i} \setminus P_{2, i}} e(\omega, z) + \frac12 \sum_{z \in P_{2, i}} e(\omega, z)
		\]
	\end{lem}
	\begin{proof}
		Suppose that $(j, \omega(j))$ is a Type Two transposition in $\omega$. Then there is some $z \in T_{2, i}$ such that $\omega$ satisfies $z$ with $j$ or $\omega(j)$. Without loss of generality, $\omega$ satisfies $z$ with $j$. It is immediate that $\omega$ satisfies $z$ with $\omega(j)$ if and only if $z \in P_{2, i}$. Call $(j, \omega(j))$ periodic if $\omega$ satisfies a sequence in $P_{2, i}$ with $j$ or $\omega(j)$, otherwise call $(j, \omega(j))$ aperiodic.
		
		\lref{l:sat} implies that there is a bijection from the aperiodic Type Two transpositions in $\omega$ to the pairs $(z, j)$ where $z \in T_{2, i} \setminus P_{2, i}$ and $j \in \F_q$ is such that $\omega$ satisfies $z$ with $j$. Hence, the number of aperiodic Type Two transpositions in $\omega$ is 
		\[
		\sum_{z \in T_{2, i} \setminus P_{2, i}} e(\omega, z).
		\]
		Similarly, there is a $1:2$ relation between the periodic Type Two transpositions in $\omega$ and the pairs $(z, j)$ where $z \in P_{2, i}$ and $j \in \F_q$ is such that $\omega$ satisfies $z$ with $j$. Hence, the number of periodic Type Two transpositions in $\omega$ is 
		\[
		\frac12 \sum_{z \in P_{2, i}} e(\omega, z).
		\]
		The lemma follows.		
	\end{proof}
	
	We are now ready to start counting transpositions in permutations in $\Gamma$. We first consider Type One transpositions. The following two lemmas can be proved by following the proof of~\cite[Lemma $4.1$]{quadcycs}.
	
	\begin{lem}\label{l:t1r}
		Let $\omega = \omega[a, b] \in \Gamma$ with $a \in \R_q$.
		\begin{enumerate}[(i)]
			\item $\omega$ satisfies $(-1, -1, -1, 1)$ if and only if 
			\[
			2(b-1)(a-b) \in \R_q \text{ and } \{(2ab-a-b)(a-b), b(a+b)(b-1)(a-b), b(a+b-2)(a-b)\} \subseteq \N_q.
			\]
			\item $\omega$ satisfies $(-1, -1, 1, -1)$ if and only if 
			\[
			2a(1-b)(a-b) \in \R_q \text{ and } \{(a+b)(1-b)(a-b), b(a+b-2ab)(a-b), (2-a-b)(a-b)\} \subseteq \N_q.
			\]
			\item $\omega$ satisfies $(-1, 1, 1, 1)$ if and only if 
			\[
			\{(a+b-2)(a-b), (a-1)(a+b)(a-b), a(2ab-a-b)(a-b)\} \subseteq \R_q \text{ and } 2b(a-1)(a-b) \in \N_q.
			\]
			\item $\omega$ satisfies $(1, -1, 1, 1)$ if and only if 
			\[
			\{a(2-a-b)(a-b), (a+b-2ab)(a-b), a(1-a)(a+b)(a-b)\} \subseteq \R_q \text{ and } 2(1-a)(a-b) \in \N_q.
			\]
		\end{enumerate}
	\end{lem}
	
	\begin{lem}\label{l:t1nr}
		Let $\omega = \omega[a, b] \in \Gamma$ with $a \in \N_q$.
		\begin{enumerate}[(i)]
			\item $\omega$ satisfies $(-1, -1, -1, 1)$ if and only if 
			\[
			a(1-a)(a+b)(a-b) \in \R_q \text{ and } \{a(2-a-b)(a-b), 2(1-a)(a-b), (a+b-2ab)(a-b)\} \subseteq \N_q.
			\]
			\item $\omega$ satisfies $(-1, -1, 1, -1)$ if and only if 
			\[
			(1-b)(a+b)(a-b) \in \R_q \text{ and } \{2a(1-b)(a-b), (2-a-b)(a-b), b(a+b-2ab)(a-b)\} \subseteq \N_q.
			\]
			\item $\omega$ satisfies $(-1, 1, 1, 1)$ if and only if 
			\[
			\{a(2ab-a-b)(a-b), 2b(a-1)(a-b), (a+b-2)(a-b)\} \subseteq \R_q \text{ and } (a-1)(a+b)(a-b) \in \N_q.
			\]
			\item $\omega$ satisfies $(1, -1, 1, 1)$ if and only if 
			\[
			\{(2ab-a-b)(a-b), b(a+b-2)(a-b), 2(b-1)(a-b)\} \subseteq \R_q \text{ and } b(b-1)(a+b)(a-b) \in \N_q.
			\]
		\end{enumerate}
	\end{lem}
	
	We can use \lref{l:t1r} and \lref{l:t1nr} in combination with \lref{l:t1count} to count the number of Type One transpositions a permutation in $\Gamma$ can contain.
	
	\begin{lem}\label{l:t1}
		Let $\omega \in \Gamma$. The number of Type One transpositions that $\omega$ contains is either $0$ or $2$.
	\end{lem}
	\begin{proof}
		Let $(a, b) \in \F_q^2$ be such that $\omega = \omega[a, b]$. We prove the result for when $a \in \R_q$, the claim for when $a \in \N_q$ can be handled using analogous arguments. Suppose that $\omega$ satisfies a sequence in $T_1$. We will show that $\omega$ satisfies exactly two sequences in $T_1$. The lemma will then follow from \lref{l:t1count}.
		Suppose that $\omega$ satisfies $(-1, -1, -1, 1)$. Define $f_1 = 2(b-1)(a-b)$, $f_2 = (a+b-2ab)(a-b)$, $f_3 = (a+b)(b-1)(a-b)$, and $f_4 = (a+b-2)(a-b)$. By \lref{l:t1r}, $f_1 \in \R_q$ and $\{-f_2, bf_3, bf_4\} \subseteq \N_q$. From \lref{l:t1r}, we know that $\omega$ satisfies $(-1, -1, 1, -1)$ if and only if $-af_1 \in \R_q$ and $\{bf_2, -f_3, -f_4\} \subseteq \N_q$. Recall that $b \in \R_q$, since $(a, b)$ is valid and $a \in \R_q$. Therefore, $\omega$ satisfies $(-1, -1, 1, -1)$ if and only if $-1 \in \R_q$, which is true if and only if $q \equiv 1 \bmod 4$. Similarly, $\omega$ satisfies $(1, -1, 1, 1)$ if and only if $q \equiv 3 \bmod 4$. If $\omega$ satisfies $(-1, 1, 1, 1)$, then by \lref{l:t1r} we must simultaneously have $f_1 \in \R_q$ and $b(a-1)f_1/(b-1) \in \N_q$. But $b(a-1)/(b-1) \in \R_q$, since $(a, b)$ is valid. Thus, $\omega$ cannot satisfy $(-1, 1, 1, 1)$. We have shown that if $\omega$ satisfies $(-1, -1, -1, 1)$, then $\omega$ satisfies precisely two sequences in $T_1$. Similar arguments reach the same conclusion if we assume that $\omega$ satisfies a sequence in $\{(-1, -1, 1, -1), (-1, 1, 1, 1), (1, -1, 1, 1)\}$.	 
	\end{proof}
	
We next consider Type Two transpositions in permutations in $\Gamma$.
	Expanding on the proof method of~\cite[Lemma $4.2$]{quadcycs}, we obtain the following lemma.
	
	\begin{lem}\label{l:t2}
		Let $\omega = \omega[a, b] \in \Gamma$ with $b \neq -a$. If $\omega$ contains a Type Two transposition, then one of the following is true:
		\begin{enumerate}[(i)]
			\item $b=2-a \in \R_q$ and $\omega$ satisfies $(-1, -1, 1, 1)$ and no other sequence in $T_{2, 1}$,
			\item $b=2-a \in \N_q$ and $\omega$ satisfies $(-1, 1, 1, -1)$ and no other sequence in $T_{2, 2}$,
			\item $b=a/(2a-1) \in \R_q$ and $\omega$ satisfies $(-1, 1, 1, -1)$ and no other sequence in $T_{2, 1}$,
			\item $b=a/(2a-1) \in \N_q$ and $\omega$ satisfies $(-1, -1, 1, 1)$ and no other sequence in $T_{2, 2}$.
		\end{enumerate}
	\end{lem}
	
	Using the same proof method of~\cite[Lemma $4.2$]{quadcycs} applied to permutations $\omega[a, -a] \in \Gamma$, we can obtain the following lemma.
	
	\begin{lem}\label{l:t2-a}
		Let $\omega = \omega[a, -a] \in \Gamma$. If $\omega$ contains a Type Two transposition, then one of the following is true:
		\begin{enumerate}[(i)]
		\item $a \in \R_q$ and $\omega$ satisfies $(-1, 1, -1, 1)$ or $(1, -1, 1, -1)$ and no other sequence in $T_{2, 3}$,
		\item $a \in \N_q$ and $\omega$ satisfies $(-1, -1, -1, -1)$ or $(1, 1, 1, 1)$ and no other sequence in $T_{2, 3}$.
		\end{enumerate}
	\end{lem}
	
	\lref{l:t2} tells us that if $\omega[a, b] \in \Gamma$ contains a Type Two transposition, then $b \in \{-a, 2-a, a/(2a-1)\}$. Also, as noted in~\cite{quadcycs}, if $(a, b) \in \F_q^2$ is a valid pair and $b \in \{-a, 2-a, a/(2a-1)\}$, then $q \equiv 1 \bmod 4$. We can use \lref{l:t2count}, \lref{l:t2} and \lref{l:t2-a} in combination with \tref{t:Weil}~\cite{Weil} and \tref{t:quadWeil}~\cite{quadWeil} below to count the number of Type Two transpositions that a permutation in $\Gamma$ can contain. By $\F_q[x]$ we mean the set of polynomials over $\F_q$.
	
	\begin{thm}\label{t:Weil}
		Let $f\in \F_q[x]$ be a polynomial of degree $n > 0$ that has $n$
		distinct roots in $\F_q$. Then,
		\begin{equation*}\label{e:weil}
			\left\vert \displaystyle\sum_{c\smash{\in \F_q}} \chi\big(f(c)\big) \right\vert \leq (n - 1)q^{1/2}.
		\end{equation*}
	\end{thm}
	
	\begin{thm}
		\label{t:quadWeil}
		Let $f \in \F_q[x]$ be a quadratic polynomial with two distinct roots in $\F_q$. Then,
		\begin{equation*}
			\left\vert \displaystyle\sum_{c \in \F_q} \chi\big(f(c)\big) \right\vert \leq 1.
		\end{equation*}
	\end{thm}
	
	\begin{lem}\label{l:t22-aa/(2a-1)}
		Let $\omega = \omega[a, b] \in \Gamma$ and let $M$ be the number of Type Two transpositions that $\omega$ contains. If $b \not\in \{-a, 2-a, a/(2a-1)\}$, then $M=0$. Otherwise,	
		\[
		\frac{q-11q^{1/2}-38}{16} \leq M \leq \frac{q+11q^{1/2}+38}{16}.
		\] 
	\end{lem}
	\begin{proof}
		If $b \not\in \{-a, 2-a, a/(2a-1)\}$, then the result follows immediately from \lref{l:t2}. Consider when $b \in \{2-a, a/(2a-1)\}$. There are four cases to consider, depending on whether $b = 2-a$ or $b=a/(2a-1)$ and whether $b \in \R_q$ or $b \in \N_q$. There is minimal difference in the arguments used for each case, so we will only consider when $b=2-a \in \R_q$. 
		
		Let $z = (-1,-1,1,1)$. \lref{l:t2count} and \lref{l:t2} together tell us that the number of Type Two transpositions that $\omega$ contains is $e(\omega, z)$. 
		Define
		\[
		V = \left\{j \in \F_q : \{j, (j+a-2)(2-a)^{-1}\} \subseteq \N_q \text{ and } \{j+a-1, a^{-1}(j-1)\} \subseteq \R_q\right\}.
		\]
		Using \eref{e:rp} and \dref{d:satisfy}, it is easy to determine that $\omega$ satisfies $z$ with $j$ if and only if $j \in V$. Thus, $e(w, z) = |V|$. Define linear polynomials $\ell_1$, $\ell_2$, $\ell_3$, and $\ell_4$ over $\F_q$ by $\ell_1(x)=x$, $\ell_2(x) = (x+a-2)(2-a)^{-1}$, $\ell_3(x) = x+a-1$, and $\ell_4(x) = a^{-1}(x-1)$. Also define
		\[
		W(x) = (1-\chi(\ell_1(x)))(1-\chi(\ell_2(x)))(1+\chi(\ell_3(x)))(1+\chi(\ell_4(x)))
		\]
		and
		\[
		S = \sum_{x \in \F_q} W(x).
		\]
		If $x \in V$, then $W(x) = 16$. If $x \in \{0, 2-a, 1-a, 1\}$, then $-8 \leq W(x) \leq 8$. If $x \in \F_q \setminus (V \cup \{0, 2-a, 1-a, 1\})$, then $W(x)=0$. Hence,
		\begin{equation}\label{e:S}
		16|V|-32 \leq S \leq 16|V|+32. 
		\end{equation}
		Expanding $W$ and using the fact that $\chi$ is a homomorphism on $\F_q^*$, we can write $S$ as a sum consisting of $\binom{4}{k}$ terms of the form $\sum_{x \in \F_q} \chi(\pm K(x))$ where $K$ is the product of $k$ distinct factors in $\{\ell_1, \ell_2, \ell_3, \ell_4\}$ for each $k \in \{0, 1, 2, 3, 4\}$. Since $(a, b)$ is valid, it follows that $\{0, 1\} \cap \{a, b\} = \emptyset$. So, $|\{0, 2-a, 1-a, 1\}|=4$. Thus, if $K$ is the product of $k$ distinct factors in $\{\ell_1, \ell_2, \ell_3, \ell_4\}$ for some $k \in \{0, 1, 2, 3, 4\}$, then $K$ has $k$ distinct roots. Therefore, we can apply \tref{t:Weil} or \tref{t:quadWeil} to each term of the form $\sum_{x \in \F_q} \chi(\pm K(x))$ with $K \neq 1$. Doing this, we obtain the following inequalities.
		\[
		q-11q^{1/2}-6 \leq S \leq q+11q^{1/2}+6.
		\] 
		Combining this with \eref{e:S} gives the result.
		
		We now deal with the case where $b=-a$. 
		We will assume that $a \in \R_q$, similar arguments can be used to deal with the case where $a \in \N_q$. Let $z' = (-1, 1, -1, 1)$ and let $z^* = (1, -1, 1, -1)$. \lref{l:t2count} and \lref{l:t2-a} together imply that the number of Type Two transpositions in $\omega$ is $(e(\omega, z')+e(\omega, z^*))/2$. First, consider the sequence $z'$. The quantity $e(\omega, z')$ is equal to the cardinality of the set
		\[
		U = \left\{j \in \F_q : \{j, 1-a-j\} \subseteq \N_q \text{ and } \{-a^{-1}(j+a), a^{-1}(j-1)\} \subseteq \R_q\right\}.
		\]
		Using the same method as above we can determine that 
		\[
		\frac{q-11q^{1/2}-38}{16} \leq |U| \leq \frac{q+11q^{1/2}+38}{16}. 
		\]
		Similarly, we can show that $e(\omega, z^*)$ is between $(q-11q^{1/2}-38)/16$ and $(q+11q^{1/2}+38)/16$. The lemma follows.	
	\end{proof}
	
	The way that \tref{t:Weil} and \tref{t:quadWeil} were utilised in the proof of \lref{l:t22-aa/(2a-1)} is by now a standard technique (see, e.g.,~\cite{quadcycs, nu4, mnaq, numquad, maxnonass}). Throughout this paper, there will be many results whose proofs rely on this technqiue. For the sake of brevity, we will often not give all the details when we prove such results. 
	
	Finally, we must address the permutations $\omega = \omega[a, b] \in \Gamma$ for which $\omega_0$ or $\omega_a$ are transpositions. By combining \cite[Lemma $4.4$]{quadcycs} and \cite[Lemma $4.5$]{quadcycs} we obtain the following result.
	
	\begin{lem}\label{l:o0oa}
		Let $\omega = \omega[a, b] \in \Gamma$. If $\omega_0$ or $\omega_a$ is a transposition, then $b \in \{-a, 2-a, a/(2a-1)\}$.
	\end{lem} 
	
	We can now prove the following result regarding the number of transpositions in a permutation in $\Gamma$.
	
	\begin{lem}\label{l:transcount}
		Let $\omega = \omega[a, b] \in \Gamma$ and let $M$ be the number of transpositions that $\omega$ contains. If $b \not\in \{-a, 2-a, a/(2a-1)\}$, then $M \in \{0, 2\}$. Otherwise,
		\[
		\frac{q-11q^{1/2}-38}{16} \leq M \leq \frac{q+11q^{1/2}+70}{16}.
		\]
	\end{lem}
	\begin{proof}
		If $b \not\in \{-a, 2-a, a/(2a-1)\}$, then the result follows by combining \lref{l:t1}, \lref{l:t22-aa/(2a-1)}, and \lref{l:o0oa}. Now suppose that $b \in \{-a, 2-a, a/(2a-1)\}$. Then $0 \in \{a+b, 2ab-a-b, 2-a-b\}$ and so \lref{l:t1r} and \lref{l:t1nr} imply that $\omega$ does not contain a Type One transposition. Therefore, $\omega$ contains at most two transpositions that are not of Type Two. The result now follows from \lref{l:t22-aa/(2a-1)}.
	\end{proof}
	
	We are now able to prove \tref{t:quadinterc}.
	
	\begin{proof}[Proof of \tref{t:quadinterc}]
		Since we have dealt with the case where $a=b$ in \sref{s:intro}, we will assume that $a \neq b$.
		We first note that $b \in \{-a, 2-a, a/(2a-1)\}$ if and only if $a \in \{-b, 2-b, b/(2b-1)\}$. 
		Let $\omega = \omega[a, b] \in \Gamma$. Let $M$ be the number of transpositions in $\omega$ and let $M'$ be the number of transpositions in $\omega[b, a]$. By \lref{l:quadrowperms}, if $q \equiv 3 \bmod 4$, then $N = Mq(q-1)/2$ and if $q \equiv 1 \bmod 4$, then $N = (M+M')q(q-1)/4$. 
		
		First, consider when $b \not\in \{-a, 2-a, a/(2a-1)\}$, so that any transposition in $\omega$ is of Type One and $\{M, M'\} \subseteq \{0, 2\}$. From~\cite[Lemma $4.1$]{quadcycs}, we know that $\omega$ contains a transposition if and only if \eref{e:n2cond} is satisfied, which is true if and only if
		\[
		(2ba-b-a)(b+a)(b-1) \in \R_q \text{ and } \{2(b+a-2)(b-1), 2b(b+a)\} \subseteq \N_q.
		\] 
		Therefore, $\omega$ contains a transposition if and only if $\omega[b, a]$ contains a transposition. Thus $M=M'$ and so $N \in \{0, q(q-1)\}$. The remaining claims for this case follow from \tref{t:n2}.
		
		Now consider when $b \in \{-a, 2-a, a/(2a-1)\}$. By \lref{l:transcount}, both $M$ and $M'$ lie between $(q-11q^{1/2}-38)/16$ and $(q+11q^{1/2}+70)/16$, which proves the required bound on $N$. 
	\end{proof}
	
	\section{Autotopisms with components in $\mathbf{\textbf{AGL}_d(p)}$}\label{s:prelim}
	
	We now turn our attention to proving \tref{t:quadatp}. Some of the arguments that we use in \sref{s:prelim}, \sref{s:not2trans}, and \sref{s:2trans} to help prove \tref{t:quadatp} do not work for `small' values of $q$. Since \tref{t:quadatp} is easy to verify using a computer if $q \leq 23$, we will assume, for \sref{s:prelim}, \sref{s:not2trans}, and \sref{s:2trans}, that $q > 23$. For these sections, we will also fix the following definitions and notation. Let $(a, b) \in \F_q^2$ be valid with $a \neq b$, let $\Q = \Q_{a, b}$, let $* = *_{a, b}$, let $\L = \L[a, b]$, and let $\varphi = \varphi[a, b]$. 

	The goal of this section is to prove \lref{l:affatp} below, which states that any autotopism of $\Q$ whose components are all elements of $\AGL_d(p)$ is an automorphism of $\Q$. \lref{l:affatp} is a key ingredient in the proof of \tref{t:quadatp}. 
	
	\begin{lem}\label{l:affatp}
		Let $(\alpha, \beta, \gamma) \in \Atp(\Q)$ and suppose that $\{\alpha, \beta, \gamma\} \subseteq \AGL_d(p)$. Then $\alpha=\beta=\gamma$.
	\end{lem}
	\begin{proof}		
		Since $\gamma(x*y) = \alpha(x) * \beta(y)$, it follows that
		\begin{equation}\label{e:isotopyeqn}
			\gamma(x+\varphi(y-x)) = \alpha(x) +\varphi(\beta(y)-\alpha(x)).
		\end{equation} 
		By applying the automorphism $\tau_{-\alpha(0)}$ of $\Q$ to $(\alpha, \beta, \gamma)$, we may assume that $\alpha$ fixes $0$. Thus, $\gamma \circ \varphi = \varphi \circ \beta$ by \eref{e:isotopyeqn}. Since $\{\beta, \gamma\} \subseteq \AGL_d(p)$, there exist $\{f, g\} \subseteq \GL_d(p)$ and $\{u, v\} \subseteq \F_q$ such that $\gamma = \tau_u \circ f$ and $\beta = \tau_v \circ g$. Our first goal is to show that $u=v=0$. 
		
		Since $\gamma \circ \varphi = \varphi \circ \beta$, for all $x \in \F_q$ we have $f(\varphi(x))+u=\varphi(g(x)+v)$ and so $f(x) = \varphi(g(\varphi^{-1}(x))+v)-u$. Setting $x=0$ shows that $u=\varphi(v)$. Suppose that $\{x, y\} \subseteq \F_q$ is such that $\chi(x) = \chi(y) = \chi(x+y) = \chi(a)$. Then the value of $f(x+y)$ is
		\[
		\begin{cases}
			ag(a^{-1}x)+ag(a^{-1}y)+av-\varphi(v) & \text{if } \chi(g(a^{-1}x)+g(a^{-1}y)+v) = 1,\\
			bg(a^{-1}x)+bg(a^{-1}y)+bv-\varphi(v) & \chi(g(a^{-1}x)+g(a^{-1}y)+v) \neq 1,
		\end{cases}
		\]
		and the value of $f(x)+f(y)$ is
		\[
		\begin{cases}
			ag(a^{-1}x)+ag(a^{-1}y)+2av-2\varphi(v) & \text{if } \chi(g(a^{-1}x)+v)= 1 =\chi(g(a^{-1}y)+v), \\
			ag(a^{-1}x)+bg(a^{-1}y)+v(a+b)-2\varphi(v) & \text{if } \chi(g(a^{-1}x)+v)= 1 \neq \chi(g(a^{-1}y)+v), \\
			bg(a^{-1}x)+ag(a^{-1}y)+v(a+b)-2\varphi(v) & \text{if } \chi(g(a^{-1}x)+v) \neq 1 = \chi(g(a^{-1}y)+v), \\
			bg(a^{-1}x)+bg(a^{-1}y)+2bv-2\varphi(v) & \text{if } 1 \not\in \{\chi(g(a^{-1}x)+v), \chi(g(a^{-1}y)+v)\}.
		\end{cases}
		\]
		Since $f(x+y) = f(x)+f(y)$, we have eight possible cases to consider. 
		\begin{enumerate}[\text{Case} 1:]
			\item $\chi(g(a^{-1}x)+g(a^{-1}y)+v) = 1$, $\chi(g(a^{-1}x)+v) = 1$, and $\chi(g(a^{-1}y)+v) = 1$. Then $av=\varphi(v)$, which implies that $v \in \R_q \cup \{0\}$.
			\item $\chi(g(a^{-1}x)+g(a^{-1}y)+v) = 1$, $\chi(g(a^{-1}x)+v) = 1$, and $\chi(g(a^{-1}y)+v) \neq 1$. Then $y=ag^{-1}((bv-\varphi(v))(a-b)^{-1})$.
			\item $\chi(g(a^{-1}x)+g(a^{-1}y)+v) = 1$, $\chi(g(a^{-1}x)+v) \neq 1$, and $\chi(g(a^{-1}y)+v) = 1$. Then $x=ag^{-1}((bv-\varphi(v))(a-b)^{-1})$.
			\item $\chi(g(a^{-1}x)+g(a^{-1}y)+v) = 1$, $\chi(g(a^{-1}x)+v) \neq 1$, and $\chi(g(a^{-1}y)+v) \neq 1$. Then $y=ag^{-1}(((b-a)g(a^{-1}x)+v(2b-a)-\varphi(v))(a-b)^{-1})$. 
			\item $\chi(g(a^{-1}x)+g(a^{-1}y)+v) \neq 1$, $\chi(g(a^{-1}x)+v) = 1$, and $\chi(g(a^{-1}y)+v) = 1$. Then $y=ag^{-1}(((a-b)g(a^{-1}x)+v(2a-b)-\varphi(v))(b-a)^{-1})$.
			\item $\chi(g(a^{-1}x)+g(a^{-1}y)+v) \neq 1$, $\chi(g(a^{-1}x)+v) = 1$, and $\chi(g(a^{-1}y)+v) \neq 1$. Then $x=ag^{-1}((av-\varphi(v))(b-a)^{-1})$.
			\item $\chi(g(a^{-1}x)+g(a^{-1}y)+v) \neq 1$, $\chi(g(a^{-1}x)+v) \neq 1$, and $\chi(g(a^{-1}y)+v) = 1$. Then $y=ag^{-1}((av-\varphi(v))(b-a)^{-1})$.
			\item $\chi(g(a^{-1}x)+g(a^{-1}y)+v) \neq 1$, $\chi(g(a^{-1}x)+v) \neq 1$, and $\chi(g(a^{-1}y)+v) \neq 1$. Then $bv = \varphi(v)$, which implies that $v \in \N_q \cup \{0\}$. 
		\end{enumerate}
		Let $X$ denote the set of all elements $x \in \F_q \setminus \{ag^{-1}((bv-\varphi(v))(a-b)^{-1}), ag^{-1}((av-\varphi(v))(b-a)^{-1})\}$ such that $\chi(x)=\chi(a)$. For $x \in X$, let $Y(x)$ denote the set of $y \in \F_q$ such that $\chi(y) = \chi(x+y) = \chi(a)$ and 
		\[
		\begin{aligned}
			y \not\in \{&ag^{-1}((bv-\varphi(v))(a-b)^{-1}), ag^{-1}(((b-a)g(a^{-1}x)+v(2b-a)-\varphi(v))(a-b)^{-1}), \\& ag^{-1}(((a-b)g(a^{-1}x)+v(2a-b)-\varphi(v))(b-a)^{-1}), ag^{-1}((av-\varphi(v))(b-a)^{-1})\}.
		\end{aligned}
		\] 
		Suppose, for a contradiction, that $v \neq 0$. First, consider when $v \in \R_q$. By construction, if $x \in X$ and $y \in Y(x)$, then Case $1$ must hold. In particular, $\chi(g(a^{-1}x)+g(a^{-1}y)+v)=1 = \chi(g(a^{-1}x)+v) = \chi(g(a^{-1}y)+v)$. We can use \tref{t:Weil} and \tref{t:quadWeil} in the usual way to show that $Y(x) \neq \emptyset$ for all $x \in X$. This conclusion relies on the fact that $q > 23$. It follows that $\chi(g(a^{-1}x)+v) = 1$ for every $x \in X$. 
	
		Define $X'$ to be the set of elements $x \in \F_q \setminus \{bg^{-1}((bv-\varphi(v))(a-b)^{-1}), bg^{-1}((av-\varphi(v))(b-a)^{-1})\}$ such that $\chi(x) = -\chi(a)$. By repeating the above arguments, we can show that $\chi(g(b^{-1}x)+v) = 1$ for every $x \in X'$. Note that if $x \in X$ and $x' \in X'$, then $a^{-1}x \neq b^{-1}x'$, since $\chi(a^{-1}x) = 1$ and $\chi(b^{-1}x') = -1$. Hence, there are at least $2((q-1)/2-2) = q-5$ elements in $\F_q$ such that $\chi(g(x)+v) = 1$. However, the map $\tau_v \circ g$ is a permutation of $\F_q$, so there are exactly $(q-1)/2$ elements $x \in \F_q$ such that $\chi(g(x)+v) = 1$. This is a contradiction, since $q-5 > (q-1)/2$ because $q > 23$. Similar arguments can be used to reach a contradiction if we assume that $v \in \N_q$. Therefore, $v=0$ and so $u=\varphi(v) = 0$ also.
		
		We may now assume that $\{\alpha, \beta, \gamma\} \subseteq \GL_d(p)$. We will prove that $\alpha=\beta$. The lemma will then follow from~\cite[Lemma $5$]{symm1facs}. Let $z \in \F_q$. By \eref{e:isotopyeqn}, for any $\{x, y\} \subseteq \F_q$,
		\begin{align}
				\gamma(x+\varphi(y-x))+\gamma(z) &= \gamma(x+z+\varphi(y+z-(x+z))) \nonumber\\
				&= \alpha(x+z)+\varphi(\beta(y+z)-\alpha(x+z)) \nonumber\\
				&= \alpha(x)+\alpha(z)+\varphi(\beta(y)+\beta(z)-\alpha(x)-\alpha(z))\label{e:gam1}.
		\end{align}
		Let $k \in \F_q$ be such that $\chi(k) = \chi(\beta(z)-\alpha(z)+k) = 1$. Since $q > 23$, such a $k \in \F_q$ can be shown to exist by using \tref{t:Weil} and \tref{t:quadWeil} in the usual way. Let $\{x, y\} \subseteq \F_q$ be such that $\beta(y)-\alpha(x)=k$. From \eref{e:gam1},
		\[
		\begin{aligned}
			\gamma(x+\varphi(y-x))+\gamma(z) &= \alpha(x)+\alpha(z)+a(\beta(y)+\beta(z)-\alpha(x)-\alpha(z)) \\
			&= \alpha(x)+\varphi(\beta(y)-\alpha(x)) + \alpha(z) + a(\beta(z)-\alpha(z)) \\
			&= \gamma(x+\varphi(y-x)) + \alpha(z)(1-a)+a\beta(z).
		\end{aligned}
		\]
		Hence, $\gamma(z) = \alpha(z)(1-a)+a\beta(z)$. Now let $\{x, y\} \subseteq \F_q$ be such that $\chi(\beta(y)-\alpha(x)) = \chi(\beta(y)-\alpha(x)+\beta(z)-\alpha(z))=-1$. Using the same argument as above, we deduce that $\gamma(z) = \alpha(z)(1-b)+b\beta(z)$. Therefore, $(b-a)(\alpha(z)-\beta(z))=0$ and so $\alpha(z)=\beta(z)$, since $a \neq b$. Since $z$ was arbitrary, it follows that $\alpha=\beta$ and we are done.
	\end{proof}
	
	\section{Not $\mathbf{2}$-transitive autotopism projections}\label{s:not2trans}
	
	The goal of this section is to prove the following lemma.
	
	\begin{lem}\label{l:not2trans}
		Let $i \in \{1, 2, 3\}$. If $\Atp_i(\Q)$ is not $2$-transitive, then $\Atp_i(\Q) \leq \AGL_d(p)$.
	\end{lem}
	
	We will need the following simple result.
	
	\begin{lem}\label{l:atp12trans}
		Let $i \in \{1, 2, 3\}$ and suppose that there exists some $\delta \in \Atp_i(\Q)$ satisfying $\delta(0)=0$ and $\delta(1) \in \N_q$. Then $\Atp_i(\Q)$ is $2$-transitive. 
	\end{lem}
	\begin{proof}
		Let $(x, y) \in \F_q^2$ with $x \neq y$. We will show that there is some $\eta \in \Atp_i(\Q)$ such that $\eta(0)=x$ and $\eta(1)=y$. If $\chi(y-x)=1$, then set $\eta = \tau_x \circ \lambda_{y-x}$. If $\chi(y-x)=-1$, then set $\eta = \tau_x \circ \lambda_{(y-x)\delta(1)^{-1}} \circ \delta$. 
	\end{proof}
	
	We are now ready to prove \lref{l:not2trans}.
	
	\begin{proof}[Proof of \lref{l:not2trans}]
Suppose that $\Atp_i(\Q)$ is not $2$-transitive. Let $\alpha_1 \in \Atp_i(\Q)$ and define $\alpha = \lambda_{((\tau_{-\alpha_1(0)} \circ \alpha_1)(1))^{-1}} \circ \tau_{-\alpha_1(0)} \circ \alpha_1$. Note that $\tau_{-\alpha_1(0)} \circ \alpha_1(1) \in \R_q$ by \lref{l:atp12trans} and so $\alpha \in \Atp_i(\Q)$. Also note that $\alpha(0)=0$ and $\alpha(1)=1$. Suppose that there is some $\{u, v\} \subseteq \F_q$ such that $\chi(v-u) \neq \chi(\alpha(v)-\alpha(u))$. We will assume that $v-u \in \R_q$ and $\alpha(v)-\alpha(u) \in \N_q$. To deal with the case where $v-u \in \N_q$ and $\alpha(v)-\alpha(u) \in \R_q$, we can apply the following arguments to $\alpha^{-1} \in \Atp_i(\Q)$. Let $c \in \N_q$ and let $\alpha_2 = \lambda_{c(\alpha(v)-\alpha(u))^{-1}} \circ \tau_{-\alpha(u)} \circ \alpha \circ \tau_u \circ \lambda_{v-u} \in \Atp_i(\Q)$. Then $\alpha_2(0)=0$ and $\alpha_2(1)=c$, which contradicts \lref{l:atp12trans}. Thus, $\chi(v-u) = \chi(\alpha(v)-\alpha(u))$ for every $\{u, v\} \subseteq \F_q$. We can now apply a theorem of Carlitz~\cite{Carlitz} to conclude that $\alpha \in \Aut(\F_q)$. Thus, $\alpha_1 \in \AGL_d(p)$. Since $\alpha_1 \in \Atp_i(\Q)$ was arbitrary, $\Atp_i(\Q) \leq \AGL_d(p)$. 
	\end{proof}
	
	\section{$\mathbf{2}$-transitive autotopism projections}\label{s:2trans}
	
	The goal of this section is to prove the following lemma.
	
	\begin{lem}\label{l:2trans}
		Let $i \in \{1, 2, 3\}$. If $\Atp_i(\Q)$ is $2$-transitive, then $\Atp_i(\Q) \leq \AGL_d(p)$.
	\end{lem}
	
	The group $\T_q$ is an elementary abelian regular subgroup of $\Atp_i(\Q)$. This allows us to make use of \tref{t:2transab} below, which is a special case of~\cite[Theorem $1.1$]{primgrpabeliansubgrp}. Let $v$ be a prime power and let $u$ be a positive integer. For the remainder of this section, we will consider any subgroup of $\AGL_u(v)$ to be acting naturally on $\F_{v^u}$ and we will consider any subgroup of $\PGammaL_u(v)$ to be acting naturally on the points of the projective space $\mathbf{P}^{u-1}(v)$.
	
	\begin{thm}\label{t:2transab}
		Let $X$ be a set of cardinality $n > 23$ and let $G \leq \Sym(X)$ be $2$-transitive with an elementary abelian regular subgroup. Then up to permutation isomorphism, one of the following is true: 
		\begin{enumerate}[(i)]
			\item $G \leq \AGL_u(v)$ where $v$ is a prime and $u$ is a positive integer such that $n=v^u$.
			\item $G \in \{\Sym(X), \Alt(X)\}$.
			\item $\PGL_u(v) \leq G \leq \PGammaL_u(v)$ for some prime power $v$ and positive integer $u$ such that $n = (v^u-1)/(v-1)$ is prime. 
		\end{enumerate}
	\end{thm}
	
	Let $i \in \{1, 2, 3\}$. If $\Atp_i(\Q)$ is $2$-transitive, then it satisfies one of $(i)$, $(ii)$, or $(iii)$ in \tref{t:2transab}. We first rule out the possibility that $\Atp_i(\Q)$ satisfies \tref{t:2transab}$(ii)$.
	
	\begin{lem}\label{l:atpnotaltsym}
		Let $n \geq 7$ be a positive integer, let $Q$ be a quasigroup of order $n$, and let $i \in \{1, 2, 3\}$. Then $\Atp_i(Q) \not\in \{\Sym(Q), \Alt(Q)\}$.
	\end{lem}
	\begin{proof}
		By~\cite[Theorem $3.1$]{atpbound},
		\begin{equation}\label{e:atpbound}
			|\Atp(Q)| \leq n^2\prod_{i=1}^{\lfloor \log_2(n) \rfloor} (n-2^{i-1}).
		\end{equation}
		It is easy to see that $n!/2$ is strictly larger than the right hand side of \eref{e:atpbound} since $n \geq 7$. The result then follows, since $|\Atp_i(Q)| \leq |\Atp(Q)|$.
	\end{proof}
	
	Let $i \in \{1, 2, 3\}$. Next, we rule out the possibility that $\Atp_i(\Q)$ satisfies \tref{t:2transab}$(iii)$. Suppose that $\Atp_i(\Q)$ does satisfy \tref{t:2transab}$(iii)$. So $q$ is prime, there is some prime power $v$ and some positive integer $u$ such that $(v^u-1)/(v-1) = q$, and up to permutation isomorphism, $\PGL_{u}(v) \leq \Atp_i(\Q) \leq \PGammaL_{u}(v)$. 
	To reach a contradiction, we consider two cases, depending on whether or not $v=2$. We first deal with the case where $v \neq 2$. To do this, we prove that there must be an element in $\PGammaL_u(v)$ of order $(q-1)/2 = (v^u-v)/(2(v-1))$. We then show that this is impossible. Note that $\PGammaL_u(2)$ does contain an element of order $2^{u-1}-1$ and so this method cannot be extended to include the case where $v=2$.
	
	\begin{lem}\label{l:atpelord}
		For each $i \in \{1, 2, 3\}$, there is an element in $\Atp_i(\Q)$ of order $(q-1)/2$.
	\end{lem}
	\begin{proof}
		Let $i \in \{1, 2, 3\}$ and let $\mu$ be a primitive element of $\F_q^*$. The map $\lambda_{\mu^2} \in \Atp_i(\Q)$ has order $(q-1)/2$.	
	\end{proof}
	
	Suppose that $v \neq 2$ is a prime power and $u$ is a positive integer such that $(v^u-1)/(v-1) = q$ is prime. Our next task is to rule out the existence of an element of $\PGammaL_u(v)$ with order $(v^u-v)/(2(v-1))$. To do this, we will use the following lemma regarding the order of elements in $\GL_u(v)$ (see, e.g.,~\cite{matrixord}).
	
	\begin{lem}\label{l:glord}
		Let $v$ be a prime power and let $u$ be a positive integer. Let $A \in \GL_u(v)$ and let $f \in \F_v[x]$ be the minimal polynomial of $A$ with distinct irreducible factors $f_1, f_2, \ldots, f_k$. For $i \in \{1, 2, \ldots, k\}$, let $e_i$ be the degree of $f_i$ and let $m_i$ be the multiplicity of $f_i$ in $f$. Let $m = \max\{m_i : i \in \{1, 2, \ldots, k\}\}$ and let $n = \lceil \log_r(m) \rceil$ where $r$ is the characteristic of $\F_v$. Then the order of $A$ divides $\textnormal{lcm}(v^{e_1}-1, v^{e_2}-1, \ldots, v^{e_k}-1) r^n$.
	\end{lem}
	
	\begin{lem}\label{l:pgammalord}
		Suppose that $q$ is prime and that there is a prime power $v \neq 2$ and a positive integer $u$ such that $q=(v^u-1)/(v-1)>23$. There is no element in $\PGammaL_u(v)$ of order $(q-1)/2$. 
	\end{lem}
	\begin{proof}	
		If $(u, v) \in \{(2, 4), (2, 8), (2, 16), (3, 3)\}$, then $q \leq 23$, which is false. We can use a computer to verify that $\PGammaL_3(8)$ has no element of order $36$.
		Therefore, we may assume that $(u, v) \not\in \{(2, 4), (2, 8), (2, 16), (3, 3), (3, 8)\}$. From~\cite[Lemma $3.1$]{prime}, we know that $u$ is prime and if $v=4$, then $u=2$.		
		
		Let $r$ be the characteristic of $\F_v$ and let $s = \log_r v$. The proof of this lemma has three parts. First, we prove that there is no element in $\GL_u(v)$ of order $(q-1)/(2x)$ for any divisor $x$ of $s$. Then, we prove that there is no element in $\GammaL_u(v)$ of order $(q-1)/2$. Finally, we prove that there is no element in $\PGammaL_u(v)$ with order $(q-1)/2$.
		
		\medskip
		
		Let $x$ be a divisor of $s$ and suppose, for a contradiction, that $A$ is an element of $\GL_u(v)$ of order $(q-1)/(2x)$. First, suppose that $u=2$. Then $(q-1)/(2x) = v/(2x)$ and so $v = 2^s$. It is well known that every element of $\GL_2(2^s)$ must have order dividing $2(2^s-1)$ or $2^{2s}-1$. The only way that $v/(2x) = 2^{s-1}/x$ divides $2(2^s-1)$ or $2^{2s}-1$ is if $2^{s-1}/x$ divides $2$. Therefore, $2^{s-1} \leq 2x \leq 2s$, which implies that $s \leq 4$. But this implies that $(u, v) \in \{(2, 4), (2, 8), (2, 16)\}$, which is a contradiction.
		
		Now suppose that $u \geq 3$ is odd. Let $e_1, e_2, \ldots, e_k$, $m_1, m_2, \ldots, m_k$, $m$, and $n$ be the parameters from \lref{l:glord} relating to the order of $A$ so that $(q-1)/(2x)$ divides $\textnormal{lcm}(v^{e_1}-1, v^{e_2}-1, \ldots, v^{e_k}-1) r^n$ and $m_1e_1+m_2e_2+ \ldots +m_ke_k \leq u$. Without loss of generality, $1 \leq e_1 \leq e_2 \leq \ldots \leq e_k \leq u$. 
		
		Write $2x = r^wx'$ for some integer $w \geq 0$ and some integer $x'$ coprime to $r$. We now prove some restrictions on the values that $w$ can take.		
		If $w \geq s$, then $v \leq r^w \leq 2x \leq 2\log_r v$. The only way that $v \leq 2\log_r v$ is possible is if $v=4$. But as noted above, if $v=4$, then $u=2$. Thus, $w < s$. Now suppose that $w = s-1$. First consider when $r \geq 3$, so that $2|x'$. Then $r^w = 2x/x' \leq s/(x'/2) = (w+1)/(x'/2)$, which implies that $w=0$ and $x'=2$. Now consider when $r=2$. If $x' \neq 1$, then $x' \geq 3$, since $x'$ is coprime to $r$. Thus, $2^w =2x/x' \leq 2s/3 = 2(w+1)/3$, which is a contradiction. Thus, $x'=1$. Now suppose that $w=s-2$, $r=2$, and $u=3$ and suppose that $x' \neq 1$. Then $2^w =2x/x' \leq 2s/3 = 2(w+2)/3$, which implies that $w \leq 1$. Since $x'$ is coprime to $r$, it follows that $w \neq 0$, thus $w=1$. Hence, $(u, v) = (3, 8)$, which is a contradiction. To summarise, we have proved the following restrictions on $w$.
		\begin{enumerate}[\text{Fact} $1$:]
			\item $w < s$.
			\item If $w=s-1$ and $r \geq 3$, then $x'=2$.
			\item If $w=s-1$ and $r=2$, then $x'=1$.
			\item If $w = s-2$, $r=2$, and $u=3$, then $x'=1$.
		\end{enumerate}
		Since $r^{s-w}(v^{u-1}-1)/(x'(v-1)) = (q-1)/(2x)$ divides $\textnormal{lcm}(v^{e_1}-1, v^{e_2}-1, \ldots, v^{e_k}-1) r^n$ and $x'(v-1)$ is coprime to $v$, it follows that $r^{s-w}$ divides $r^n$. Thus,
		\[
		s-w \leq n = \lceil \log_r m \rceil < \log_r m+1
		\]
		and so $m > r^{s-w-1}$. Fact $1$ implies that $r^{s-w-1}$ is an integer and so $m \geq 1+r^{s-w-1}$. Hence,
		\begin{equation}\label{e:disum}
			\sum_{i=1}^{k} e_i \leq u-\sum_{i=1}^{k} (m_i-1)e_i \leq u-e_1r^{s-w-1},
		\end{equation}
		where the second inequality in \eref{e:disum} is obtained by setting $m_1 = m \geq r^{s-w-1}+1$ and $m_i = 1$ for all $i \in \{2, 3, \ldots, k\}$. By \lref{l:glord}, there exist pairwise coprime integers $\ell_1, \ell_2, \ldots, \ell_k$ such that $\ell_i | v^{e_i}-1$ for all $i \in \{1, 2, \ldots, k\}$ and $(v^{u-1}-1)/(x'(v-1)) = \ell_1\ell_2 \cdots \ell_k$. Since $v-1|v^{e_i}-1$ for all $i \in \{1, 2, \ldots, k\}$, it follows that there are integers $t_1, t_2, \ldots, t_k$ such that $t_1t_2 \cdots t_k | v-1$ and $\ell_i | t_i(v^{e_i}-1)/(v-1)$ for all $i \in \{1, 2, \ldots, k\}$. Therefore,		
		\begin{equation}\label{e:ordineq}
			\frac{v^{u-1}-1}{x'(v-1)} = \prod_{i=1}^{k} \ell_i \Bigg| \prod_{i=1}^{k} t_i \frac{v^{e_i}-1}{v-1} \Bigg| (v-1)^{1-k} \prod_{i=1}^{k} (v^{e_i}-1).
		\end{equation}
		Also, from \eref{e:disum},
		\begin{equation}\label{e:ordineq2}
			\prod_{i=1}^{k} (v^{e_i}-1) \leq \left(\prod_{i=1}^{k} v^{e_i} \right)-1 = v^{e_1+e_2+\ldots e_k}-1 \leq v^{u-e_1r^{s-w-1}}-1.
		\end{equation}
		Furthermore, the first inequality in \eref{e:ordineq2} is strict if $k \geq 2$.  
		
		We first deal with the case where $k \geq 2$. Suppose that $e_1+e_2+\ldots +e_k \leq u-2$. Then the combination of \eref{e:ordineq} and \eref{e:ordineq2} yields,
		\[
		v^{u-1}-1 < x'(v-1)^{2-k}(v^{u-2}-1) \leq x'(v^{u-2}-1).
		\]
		This implies that $v^{u-2}(v-x') < 1-x'$ and in particular $v < x' \leq 2x \leq 2\log_r v$, which is impossible. Thus, we may assume that $e_1+e_2+\ldots +e_k \geq u-1$. By \eref{e:disum}, we know that $e_1+e_2+\ldots +e_k = u-1$ and $e_1=1=r^{s-w-1}$. So $s=w+1$. Fact $2$ and Fact $3$ together imply that $x' \in \{1, 2\}$.  
		The combination of \eref{e:ordineq} and \eref{e:ordineq2} tells us that $v^{u-1}-1 < x'(v-1)^{2-k}(v^{u-1}-1)$. Thus, $(v-1)^{k-2} < x'$ and so $k=2=x'$. Since $e_1 = 1$, it follows that $e_2=u-2$ and \eref{e:ordineq} implies that $(v^{u-1}-1) | 2(v-1)(v^{u-2}-1)$. If $v^{u-1}-1 < 2(v-1)(v^{u-2}-1)$, then $v^{u-1}-1 \leq (v-1)(v^{u-2}-1)$, which is false. So $v^{u-1}-1=2(v-1)(v^{u-2}-1)$. Thus, $v^{u-2}(2-v) = 3-2v$, which implies that $3-2v \geq 2-v$. However, this forces $v \leq 1$, a contradiction.
		
		We now deal with the case where $k=1$. Since $(v^{u-1}-1)/(x'(v-1))$ divides $v^{e_1}-1$ and $v^{u-1}-1$, it must also divide $\gcd(v^{e_1}-1, v^{u-1}-1) = v^{\gcd(u-1, e_1)}-1$. If $e_1 = u-1$, then \eref{e:disum} says that $u-1 \leq u-(u-1)r^{s-w-1}$, forcing $u \leq 2$, which is false. So $e_1 \neq u-1$ and thus $\gcd(u-1, e_1) \leq (u-1)/2$. Then \eref{e:ordineq} says that $v^{u-1}-1 \leq x'(v-1)(v^{(u-1)/2}-1) \leq x'(v^{(u+1)/2}-1)$, which implies that 
		\begin{equation}\label{e:k1}
			v^{(u+1)/2}(v^{(u-3)/2}-x') \leq 1-x'.
		\end{equation}
		If $x'=1$, then \eref{e:k1} implies that $v^{(u-3)/2} \leq 1$, so $u=3$. If $x' > 1$, then \eref{e:k1} implies that $v^{(u-3)/2} < x' \leq 2\log_r v$, which also implies that $u=3$. Thus, we may assume that $u=3$. Since $e_1 \neq u-1$, it follows that $e_1=1$. So \eref{e:disum} implies that $r^{s-w-1} \leq 2$. 
		
		First consider when $r \geq 3$. Since $r^{s-w-1} \leq 2$, Fact $1$ and Fact $2$ imply that $x'=2$. So \eref{e:ordineq} implies that $(v+1)/2 = (v^{u-1}-1)/(2(v-1))$ divides $v-1$. Since $(v+1)/2 > (v-1)/2$, it follows that $(v+1)/2=v-1$. Hence, $v=3$ and so $(u, v) = (3, 3)$, a contradiction. 
		
		Now suppose that $r=2$. Since $2^{s-w-1} \leq 2$, Fact $1$, Fact $3$, and Fact $4$ together imply that $x'=1$. Hence, \eref{e:ordineq} implies that $v+1 = (v^{u-1}-1)/(v-1)$ divides $v-1$, which is impossible. This completes the proof of the fact that no element in $\GL_u(v)$ has order $(q-1)/(2x)$.
		
		\medskip
		
		We now prove that there is no element in $\GammaL_u(v)$ of order $(q-1)/2$. Recall that $\GammaL_u(v)$ is the semidirect product $\GL_u(v) \rtimes \Aut(\F_v)$ where an element $\theta \in \Aut(\F_v)$ acts entry-wise on an element $A \in \GL_u(v)$. Suppose, for a contradiction, that there is some element $A\theta \in \GammaL_u(v)$ of order $(q-1)/2$. Let $x$ be the order of $\theta \in \Aut(\F_v)$ and note that $x|s$. Then $(A\theta)^x = B$ for some $B \in \GL_u(v)$. Therefore, we can write $(q-1)/2 = yx$ where $y$ divides the order of $B$. Let $b$ denote the order of $B$ so that $b=yb'$ for some positive integer $b'$. Then $B^{b'} \in \GL_u(v)$ has order $y = (q-1)/(2x)$, which is a contradiction.
		
		\medskip
		
		We now prove that there is no element in $\PGammaL_u(v)$ of order $(q-1)/2$. Let $I$ denote the $u \times u$ identity matrix in $\GL_u(v)$. Let $Z = \{kI : k \in \F_v\}$ be the centre of $\GL_u(v)$. Recall that $\PGammaL_u(v)$ is the quotient group $\GammaL_u(v) / Z$. For an element $A\theta \in \GammaL_u(v)$,  we will denote by $\overline{A\theta}$ the image of $A\theta$ under the projection map from $\GammaL_u(v)$ to $\PGammaL_u(v)$. Suppose, for a contradiction, that there is some element $\overline{A\theta} \in \PGammaL_u(v)$ of order $(q-1)/2$. Then there is some $k \in \F_v$ such that $(A\theta)^{(q-1)/2} = kI$. Let the order of $k$ in $\F_v$ be $y$. Then the order of $(A\theta)^y$ in $\GammaL_u(v)$ is $(q-1)/2$, which is a contradiction.
	\end{proof}
	
	Let $i \in \{1, 2, 3\}$ and suppose that $\Atp_i(\Q)$ satisfies \tref{t:2transab}$(iii)$. So, up to permutation isomorphism, $\PGL_u(v) \leq \Atp_i(\Q) \leq \PGammaL_u(v)$ for some prime power $v$ and positive integer $u$ such that $q = (v^u-1)/(v-1)$ is prime. \lref{l:atpelord} and \lref{l:pgammalord} together imply that $v = 2$. To reach a contradiction, we first prove that $\Atp_i(\Q)$ must contain an element that is the product of $(q+1)/4$ disjoint transpositions. We then use some results from \sref{s:interc} about the structure of quadratic Latin squares to prove that this is impossible. Note that $2^u-1 \equiv 3 \bmod 4$ for any positive integer $u \geq 2$. 
	
	\begin{lem}\label{l:pglord2}
		Suppose that $q$ is prime and that there is an integer $u \geq 2$ such that $q = 2^u-1$. Let $G \leq \Sym(\F_q)$ be permutation isomorphic to $\PGL_u(2)$. There is some $g \in G$ that is the product of $(q+1)/4$ disjoint transpositions.
	\end{lem}
	\begin{proof}	
		Identify the action of $\PGL_u(2)$ on the points of $\mathbf{P}^{u-1}(2)$ with $\GL_u(2)$ acting naturally on the set of non-zero elements of $\F_2^u$. Let $A \in \GL_u(2)$ be defined by
		\[
		A_{i, j} = \begin{cases}
			1 & \text{if } i=j \text{ or both } i=u \text{ and } j=1, \\
			0 & \text{otherwise}.
		\end{cases}
		\]
		Let $x = (x_1, x_2, \ldots, x_u)$ be a non-zero element of $\F_2^u$. If $x_1=0$, then $Ax=x$ and there are $2^{u-1}-1$ possibilities for $x$. If $x_1=1$, then $A^2x=x \neq Ax$ and there are $2^{u-1} = (q+1)/2$ possibilities for $x$. The lemma follows.
	\end{proof}
	
	\begin{lem}\label{l:atpord2}
		Suppose that $q \equiv 3 \bmod 4$ is prime. For each $i \in \{1, 2, 3\}$, there is no permutation in $\Atp_i(\Q)$ that is the product of $(q+1)/4$ disjoint transpositions.
	\end{lem}
	\begin{proof}
		By taking parastrophes, it suffices to prove the statement assuming that $i=1$.
		Suppose, for a contradiction, that $\Atp_1(\Q)$ contains an element $\alpha$ that is the product of $(q+1)/4$ disjoint transpositions and has $(q-1)/2$ fixed points. Let $\beta \in \Atp_2(\Q)$ and $\gamma \in \Atp_3(\Q)$ be such that $(\alpha, \beta, \gamma) \in \Atp(\Q)$. By using~\cite[Lemma $8$]{autquasi}, we deduce that the order of $\beta$ is $2$ and the order of $\gamma$ is $2$. Since $q$ is odd, it follows that $\beta$ and $\gamma$ must have fixed points. We can use~\cite[Theorem $1$]{smalllatin} to deduce that $\beta$ and $\gamma$ must have exactly $(q-1)/2$ fixed points. 
		
		Let $R$, $C$, and $S$ be the set of elements of $\F_q$ fixed by $\alpha$, $\beta$, and $\gamma$, respectively. Also let $R'$, $C'$, and $S'$ be the complements of $R$, $C$, and $S$ in $\F_q$, respectively. 
		Note that if $x \in R$ and $y \in C$, then $x*y \in S$. 
		Fix $x' \in R'$. We claim that there is a unique $y' \in C'$ such that $x'*y' \in S'$. For any $y \in C$, it is true that $x'*y \in S'$, since all symbols $x*y$ with $x \in R$ lie in $S$. There are $(q+1)/2$ symbols in $S'$ and there are $(q-1)/2$ elements $y \in C$ such that $x'*y \in S'$. Hence, there is a unique $y' \in C'$ such that $x'*y' \in S'$, as claimed. Let $y' \in C'$ be such that $\{x'*y', x'*\beta(y')\} \subseteq S$. Note that there are at least $(q-3)/2$ choices for $y'$. Since $(\alpha, \beta, \gamma) \in \Atp(\Q)$ and $\{x'*y', x'*\beta(y')\} \subseteq S$, it follows that $x'*y' = \alpha(x')*\beta(y')$ and $x'*\beta(y') = \alpha(x')*\beta^2(y') = \alpha(x')*y'$. Equivalently, the row permutation $r_{x', \alpha(x')}$ of $\L$ contains the transposition $(x'*y', x'*\beta(y'))$. Since there are at least $(q-3)/2$ choices for $y'$, it follows that $r_{x', \alpha(x')}$ contains at least $(q-3)/4$ transpositions. Recall that $b \not\in \{-a, 2-a, a/(2a-1)\}$ since $q \equiv 3 \bmod 4$. Thus, \lref{l:transcount} combined with \lref{l:quadrowperms} implies that $r_{x, \alpha(x)}$ contains at most two transpositions, which is a contradiction, since $q>23$.	
	\end{proof}
	
	By combining \tref{t:2transab}, \lref{l:atpnotaltsym}, \lref{l:atpelord}, \lref{l:pgammalord}, \lref{l:pglord2}, and \lref{l:atpord2}, we have the following result.
	
	\begin{cor}\label{c:notpgl}
		Let $i \in \{1, 2, 3\}$. If $\Atp_i(\Q)$ is $2$-transitive, then up to permutation isomorphism, $\Atp_i(\Q) \leq \AGL_d(p)$.
	\end{cor}

 We are now ready to prove \lref{l:2trans}. Let $G$ and $H$ be permutation groups with $H \leq G$ and let $g \in G$. We denote by $H^g$ the subgroup $\{g^{-1} \circ h \circ g : h \in H\} \leq G$.
	
	\begin{proof}[Proof of \lref{l:2trans}]
		Suppose that $\Atp_i(\Q)$ is $2$-transitive. From \cyref{c:notpgl},
		there is some $G \leq \AGL_d(p)$ and some $\delta \in \Sym(\F_q)$ such that $\Atp_i(\Q) = G^\delta$. Let $\T = \T_q$. Since $\T \leq \Atp_i(\Q)$, it follows that $\T^{\delta^{-1}} \leq G$. There is some $\mu \in \F_q^*$ such that $\tau_\mu \in \T \cap \T^{\delta^{-1}} \leq G$ (see, e.g.,~\cite[Lemma $4$]{transintersect}). Since $\tau_\mu \in \T^{\delta^{-1}}$, it follows that $\delta^{-1} \circ \tau_\mu \circ \delta = \tau_\nu$ for some $\nu \in \F_q^*$. Since $\T$ is normal in $\AGL_d(p)$, it follows that $\T \cap G$ is normal in $G$. Hence, $(\T \cap G)^\delta$ is normal in $\Atp_i(\Q)$. Note that $\tau_\nu \in (\T \cap G)^\delta$, since $\tau_\mu \in \T \cap G$. Since $(\T \cap G)^\delta$ is normal in $\Atp_i(\Q)$ and $\{\lambda_c : c \in \R_q\} \leq \Atp_i(\Q)$, it follows that $\tau_{c\nu} = \lambda_c \circ \tau_\nu \circ \lambda_{c^{-1}} \in (\T \cap G)^\delta$ for all $c \in \R_q$. So $|\T \cap (\T \cap G)^\delta| \geq (q+1)/2 > |\T|/2$, hence $\T =(\T \cap G)^\delta$ and so $\T^\delta = \T$. Therefore, $\delta$ is in the normaliser of $\T$ in $\Sym(\F_q)$, which is $\AGL_d(p)$. Thus, $\Atp_i(\Q) = G^\delta \leq \AGL_d(p)$.
	\end{proof}
	
	\section{Proof of main results}\label{s:quadisotop}
	
	\tref{t:quadatp} was shown to be true for quadratic quasigroups $\Q_{a, b}$ with $a=b$ in \sref{s:intro}. As mentioned at the beginning of \sref{s:prelim}, \tref{t:quadatp} is easy to verify using a computer if $q < 23$. \tref{t:quadatp} in the case where $a \neq b$ and $q \geq 23$ follows by combining \lref{l:affatp}, \lref{l:not2trans}, and \lref{l:2trans}. We now prove \tref{t:quadisotop}.
	
	\begin{proof}[Proof of \tref{t:quadisotop}]
		Let $\Q = \Q_{a, b}$ and $\Q' = \Q_{a', b'}$ be quadratic quasigroups of order $q$. Since we have already dealt with the case where $a=b$ or $a' = b'$ in \sref{s:intro}, we may assume that $a \neq b$ and $a' \neq b'$. By \tref{t:quadiso}, if $\{a, b\} = \{\theta(a'), \theta(b')\}$ for some $\theta \in \Aut(\F_q)$, then $\Q$ and $\Q'$ are isotopic. Now assume that $\Q$ and $\Q'$ are isotopic and let $(\alpha, \beta, \gamma)$ be an isotopism from $\Q$ to $\Q'$. Then $(\alpha^{-1} \circ \theta \circ \alpha, \beta^{-1} \circ \theta \circ \beta, \gamma^{-1} \circ \theta \circ \gamma) \in \Atp(\Q)$ for any $\theta \in \Aut(\Q')$. \tref{t:quadatp} implies that $\alpha^{-1} \circ \theta \circ \alpha = \beta^{-1} \circ \theta \circ \beta$ and so $\alpha \circ \beta^{-1}$ centralises $\theta$. Since $\theta \in \Aut(\Q')$ was arbitrary, it follows that $\alpha \circ \beta^{-1}$ centralises $\Aut(\Q')$. In particular, $\alpha \circ \beta^{-1}$ centralises $\T_q$ and so must be an element of $\T_q$. But $\alpha \circ \beta^{-1}$ also centralises the permutation $\lambda_c$ where $c$ is some fixed element of $\R_q \setminus \{1\}$. The only element of $\T_q$ that can centralise this permutation is $\Id$. Therefore, $\alpha = \beta$. Similar arguments show that $\beta=\gamma$ also. Hence, $\Q$ and $\Q'$ are isomorphic and the theorem follows from \tref{t:quadiso}.
	\end{proof}
	
	\section{Conclusion}\label{s:conc}
	
	An \emph{orthomorphism} of a group $(G, +)$ is a permutation $\psi$ of $G$ such that the map $x \mapsto \psi(x)-x$ is also a permutation of $G$. Suppose that $\psi$ is an orthomorphism of $G$. Define a binary operation $\cdot$ on $G$ by $x\cdot y=x+\psi(y-x)$. Then the pair $(G, \cdot)$ forms a quasigroup~\cite{orthls}. We say that $(G, \cdot)$ is \emph{generated} by $\psi$.
	
	Let $n$ and $k$ be positive integers such that $q-1=nk$. Let $\kappa$ be a primitive element of $\F_q^*$. For each $j \in \{0, 1, \ldots, n-1\}$, define $C_j = \{\kappa^{ni+j} : i \in \{0, 1, \ldots, k-1\}\}$. Then $C_j$ is a \emph{cyclotomic coset} of the unique subgroup $C_0$ in $\F_q^*$ of index $n$. For $(a_0, a_1, \ldots, a_{n-1}) \in \F_q^n$, let $\varphi = \varphi_\kappa[a_0, a_1, \ldots, a_{n-1}]$ be the map defined by
	\[
	\varphi(x) = \begin{cases}
		0 & \text{if } x=0, \\
		a_jx & \text{if } x \in C_j.
	\end{cases}
	\]
	Then $\varphi$ is a \emph{cyclotomic map of index $n$}. If $\varphi$ is cyclotomic of index $n$ and is not cyclotomic of any index $m<n$, then $\varphi$ has \emph{least index $n$}. If $n=2$, then $C_0 = \R_q$ and $C_1 = \N_q$, regardless of the choice of $\kappa$. In this case $\varphi$ is called a \emph{quadratic map}. If a cyclotomic map is also an orthomorphism, then it is called a \emph{cyclotomic orthomorphism}. If a quadratic map is also an orthomorphism, then it is called a \emph{quadratic orthomorphism}. 
	
	Quadratic quasigroups are precisely the quasigroups that are generated by a quadratic orthomorphism. Quadratic quasigroups of the form $\Q_{a, a}$ are precisely the quasigroups that are generated by a cyclotomic orthomorphism of index $1$. The main achievements of this paper can thus be stated as follows:
	\begin{equation}\label{e:atpaut}
		\begin{aligned}
			&\text{If $\Q$ is a quasigroup generated by a cyclotomic orthomorphism} \\
			&\text{of least index $2$, then every autotopism of $\Q$ is an automorphism.}
		\end{aligned}
	\end{equation}
	\begin{equation}\label{e:isotopiso}
		\begin{aligned}
			&\text{If $\Q$ and $\Q'$ are isotopic quasigroups generated by cyclotomic} \\
			&\text{orthomorphisms of least index $2$, then $\Q$ and $\Q'$ are isomorphic.}
		\end{aligned}
	\end{equation}
	
	It would be interesting to determine for which integers $n$ \eref{e:atpaut} is true if we replace $2$ by $n$ and for which integers $n$ \eref{e:isotopiso} is true if we replace $2$ by $n$. Neither \eref{e:atpaut} nor \eref{e:isotopiso} are true if we replace $2$ by $1$. Let $\Q$ and $\Q'$ be quasigroups of order $q$ that are generated by cyclotomic orthomorphisms of least index $1$. \tref{t:quadaut} and \tref{t:quadatp} together imply that there are autotopisms of $\Q$ that are not automorphisms of $\Q$. \tref{t:quadiso} and \tref{t:quadisotop} together imply that $\Q$ and $\Q'$ are isotopic, but not necessarily isomorphic. 
	
	Identify $\F_9 = \mathbb{F}_3[x]/(x^2-x-1)$. Let $\Q$ be the quasigroup generated by the cyclotomic orthomorphism $\varphi_x[-1, x+1, -x+1, x-1]$ and let $\Q'$ be the quasigroup generated by the cyclotomic orthomorphism $\varphi_x[x, x+1, -x-1, -x+1]$. Then $\Q$ and $\Q'$ are isotopic, but not isomorphic. Similarly, identify $\F_{25} = \mathbb{F}_5[x] / (x^2-x+2)$, let $\Q$ be the quasigroup generated by $\varphi_x[2, -1, -x-2, x-1, 2x-2, 2x-2]$, and let $\Q'$ be the quasigroup generated by $\varphi_x[2, -2, 2x, 1-x, x, 2-2x]$. Then $\Q$ and $\Q'$ are isotopic, but not isomorphic. Therefore, \eref{e:isotopiso} is false if we replace $2$ by $4$ or $6$. For $n \in \{3, 5\}$, there are no two quasigroups $\Q$ and $\Q'$ of order less than $50$ generated by cyclotomic orthomorphisms of least index $n$ that are isotopic but not isomorphic. For $n \in \{3, 4, 5\}$, there is no quasigroup of order less than $50$ generated by a cyclotomic orthomorphism of least index $n$ that has an autotopism that is not an automorphism. The only quasigroup of order less than $50$ generated by a cyclotomic orthomorphism of least index $6$ that has an autotopism that is not an automorphism is the quasigroup $\Q$ of order $7$ generated by $\varphi_3[3, 3, 6, 3, 2, 4]$. Explicitly, $\Aut(\Q)=\T_7$ and $\Atp(\Q) \cong \PGL_3(2)$. However, we note that a cyclotomic orthomorphism of index $q-1$ over $\F_q$ is, in some sense, a degenerate cyclotomic orthomorphism, since every orthomorphism over $\F_{q}$ that fixes $0$ is cyclotomic of index $q-1$. 
	
	The only known quasigroups that violate \eref{e:isotopiso} when $2$ is replaced by some other index are not of prime order. Thus, it would be interesting to determine for which integers $n$ \eref{e:isotopiso} is true when replacing $2$ by $n$ if we impose that the orders of $\Q$ and $\Q'$ are prime. It would also be interesting to determine for which integers $n$ \eref{e:atpaut} is true when replacing $2$ by $n$ if we impose that the orders of $\Q$ and $\Q'$ are prime.
	
	\section*{Acknowledgements}
	
	I would like to thank Ian Wanless, Ale\v{s} Dr\'{a}pal, and Melissa Lee for many helpful discussions.
	
	\printbibliography
	
\end{document}